\documentclass[a4paper, 11pt]{amsart}
\usepackage[latin1]{inputenc}
\usepackage[T1]{fontenc}
\usepackage[english]{babel}
\usepackage{amssymb}
\usepackage{amsmath}
\usepackage{amsthm}
\usepackage{amscd}
\usepackage{amsfonts}
\usepackage{stmaryrd}
\usepackage{pb-diagram}
\usepackage{epic,eepic,epsfig}
\usepackage{a4wide}
\usepackage{nextpage}
\usepackage{fancyhdr}

\newtheorem{prop}{Proposition}[section]
\newtheorem{cor}[prop]{Corollary}

\newtheorem{lem}[prop]{Lemma}
\newtheorem{thm}[prop]{Theorem}

\theoremstyle{definition}
\newtheorem{rem}[prop]{Remark}

\newtheorem{red}[prop]{Reduction}

\newcommand{\lcm} {\mathop{\mathrm{lcm}}}
\newcommand{\GL}{\mathop{\mathrm{GL}}}
\newcommand{\PGL}{\mathop{\mathrm{PGL}}}
\newcommand{\into}{\hookrightarrow}

\newcommand{\BC} {\mathbb{C}}
\newcommand{\BF} {\mathbb{F}}
\newcommand{\BP} {\mathbb{P}}
\newcommand{\BQ} {\mathbb{Q}}
\newcommand{\BR} {\mathbb{R}}
\newcommand{\BZ} {\mathbb{Z}}

\newcommand{\BT}{\mathcal{BT}}
\newcommand{\CF}{\mathcal{F}}
\newcommand{\CO}{\mathcal{O}}

\newcommand{\Fm}{\mathfrak{m}}

\newcommand{\Cinf}{\BC_\infty}
\newcommand{\Finf}{F_\infty}

\newcommand{\hTag}{h_{\mathrm{Tag}}}

\begin{document}

\title{Heights and Isogenies of Drinfeld modules}
\author{{F}lorian {B}reuer, {F}abien {P}azuki, {M}ahefason {H}eriniaina {R}azafinjatovo}
\address{Florian Breuer. School of Mathematical and Physical Sciences,
University of Newcastle, Newcastle, Australia.}
\email{florian.breuer@newcastle.edu.au}
\address{Fabien Pazuki. University of Copenhagen, Institute of Mathematics, Universitetsparken 5, 2100 Copenhagen, Denmark, and Universit\'e de Bordeaux, IMB, 351, cours de la Lib\'eration, 33400 Talence, France.}
\email{fpazuki@math.ku.dk}
\address{Mahefason Heriniaina Razafinjatovo.
University of Antananarivo,
Madagascar.}
\email{razafpouf@univ-antananarivo.mg}
\maketitle

\begin{abstract}
We provide explicit bounds on the difference of heights of isogenous Drinfeld modules. We derive a finiteness result in isogeny classes. In the rank 2 case, we also obtain an explicit upper bound on the size of the coefficients of modular polynomials attached to Drinfeld modules.
\end{abstract}

{\flushleft
 }

\begin{center}
---------
\end{center}

\thispagestyle{empty}

\section{Introduction}
Let $E$ and $E'$ be two elliptic curves over a number field, linked by an isogeny $f : E \to E'$. Can we compare their heights? In the case of the Faltings height, a classical result \cite{Falt, Raynaud} states that 
\begin{equation}\label{rayeq}
|h_{\mathrm{Falt}}(E) - h_{\mathrm{Falt}}(E')| \leq \frac{1}{2}\log\deg f.
\end{equation}
A more elementary height is $h(j(E))$, the Weil height of the $j$-invariant of $E$. For this height, the second author \cite{Paz} proved 
\begin{equation}\label{pazeq}
|h(j(E)) - h(j(E'))| \leq 9.204 + 12 \log\deg f.
\end{equation}
The proof of (\ref{pazeq}) involves modifying the Faltings height at the infinite places so that the result can be deduced from (\ref{rayeq}).

In the present paper, we consider function field analogues of these results. Consider two Drinfeld $\BF_q[t]$-modules of rank $r \geq 2$ linked by an isogeny $f : \phi \to \phi'$. There are several notions of height of a Drinfeld module; the best analogue of the Faltings height was defined by Taguchi \cite{Tag}, who also proved a variant of (\ref{rayeq}) for Drinfeld modules (see Lemma \ref{TaguchiIsog} below).

For the more elementary height $h_G$ associated to the coefficients of a Drinfeld module, we prove an analogue of (\ref{pazeq}) of the form
\begin{equation}\label{oureq}
|h_G(\phi') - h_G(\phi)| \; \leq\;
\log_q\deg f + \left(\frac{q}{q-1} - \frac{q^r}{q^r-1}\right).
\end{equation} 

Our basic approach is somehow similar to that in \cite{Paz}, with some natural changes: we use analytic estimates based on the technology developed by Gekeler \cite{Gek0, Gek1, GekCrelle}, notably the fundamental domain for the moduli space of Drinfeld modules, Bruhat-Tits buildings, and we conclude by an invocation of Taguchi's Isogeny Lemma. Combined with a deeper result of David-Denis \cite{DD}, this allows us to give a new proof of the finiteness of isomorphism classes of Drinfeld modules over a global function field within each isogeny class, see Corollary \ref{finiteness}. A variant of (\ref{oureq}) in the special case $r=2$ allows us to deduce explicit estimates on the height of Drinfeld modular polynomials in the rank two case, see Proposition \ref{Drinfeld modular} and equation (\ref{Drinfeld modular asymp}) below.

The layout of this paper is as follows. In Section \ref{sec:heights} we define heights associated to Drinfeld modules. Our main results are stated in Section \ref{sec:main}.

In Section \ref{sec:Taguchi} we introduce the notion of a reduced Drinfeld module and define Taguchi's height. Our main results are then proved in Section \ref{sec:analytic}, where we compute various analytic estimates. Finally, we deduce the upper bound on the coefficients of Drinfeld modular polynomials (in the rank two case) in Section \ref{sec:modular}.

\subsection*{Acknowledgement.} The authors are grateful to Fu-Tsun Wei for helpful discussions and the anonymous referee for constructive feedback. The authors thank the CNRS and the IRN GANDA for the support.
The first author thanks the Alexander-von-Humboldt Foundation for partial support. The second author thanks the projects ANR-14-CE25-0015 Gardio and ANR-17-CE40-0012 Flair.

\section{Heights of Drinfeld modules}\label{sec:heights}

\subsection{Places}
Let $A=\BF_q[t]$ and $F=\BF_q(t)$. To each place $v$ of $F$ we associate an absolute value $|\cdot|_v$ normalized as follows. A place of $F$ corresponding to a monic irreducible polynomial $P\in A$ is called a finite place, and we have $|x|_v = q^{-\deg(P) v_P(x)}$ for $x\in F$. There is one more place, denoted $\infty\in M_F$, with 
$|x|_\infty = q^{\deg(x)}$. 

For a finite extension $K/F$ we denote by $M_K$ the set of places of $K$. A place $v\in M_K$ is called infinite if it is an extension of $\infty$, otherwise it is called finite. The set of finite and infinite places of $K$ are denoted by $M_K^f$ and $M_K^\infty$, respectively.

To each place $v\in M_K$ we associate its absolute value normalized so that, for every $x\in F$ we have $|x|_v = |x|_w$, where $w\in M_F$ lies beneath $v$.

To each place $v\in M_K$ we also associate the ramification index $e_v$ (so $|K|_v \subset q^{\frac{1}{e_v}\BZ}$), the residual degree $f_v$ and the local degree $n_v = [K_v:F_v] = e_v f_v$.

We have the following two important properties:
\begin{itemize}
	\item {Product Formula:} For every $x\in K$, $\displaystyle \prod_{v\in M_K} |x|_v^{n_v} = 1$.
	\item {Extension Formula:} For every $w\in M_F$ we have $\displaystyle{[K:F] = \sum_{v|w}n_v}$.
\end{itemize}

Note that in articles like \cite{DD} and \cite{Tag}, the absolute values are normalized differently; the exponent $e_v$ is included in $|x|_v$, so in their situation the product formula holds with $n_v$ replaced by $f_v$.

Finally, for the remainder of this article $\log$ always means the logarithm to the base $q$.
\bigskip

We will associate to a Drinfeld module a number of different heights. Every height $h$ will be decomposed into a sum of local heights, $\displaystyle{ h = \frac{1}{[K:F]}\sum_{v\in M_K} h^v}$. We also write $\displaystyle{h^f = \frac{1}{[K:F]}\sum_{v \in M_K^f}h^v}$ and $\displaystyle{h^\infty = \frac{1}{[K:F]}\sum_{v\in M_K^\infty}h^v}$ for the finite and infinite components, respectively.

\subsection{Na\"ive heights}

Let $\phi$ be a Drinfeld $\BF_q[t]$-module of rank $r$ over $K$. We assume throughout this paper that $r\geq 2$ and that our Drinfeld modules are of generic characteristic. Then $\phi$ is characterised by 
\[
\phi_t(X) = tX + g_1X^q + g_2X^{q^2} + \cdots + g_rX^{q^r},\qquad g_i\in K, \; g_r\neq 0.
\]
We refer to the $g_1,\ldots,g_r$ as the {\em coefficients} of $\phi$.

Let $d=\lcm\{q-1,q^2-1,\ldots,q^r-1\}$. 
For $k=1,2,\ldots, r$, set
\begin{equation}
j_k := \frac{g_k^{d/(q^k-1)}}{g_r^{d/(q^r-1)}}\in K, \quad\text{and}\quad J=J(\phi)=(j_1,j_2,\ldots,j_r).
\end{equation}
Clearly $j_r=1$. These are isomorphism invariants of $\phi$. 

\begin{rem}\label{Potemine}
These invariants differ from those defined by Potemine in \cite[(2.5)]{Pot} in their exponents: we have chosen exponents such that each $j_k$ has the same denominator, whereas Potemine used the least integer exponents for each $j_k$. Nevertheless, it follows from \cite[Theorem 2.2]{Pot} that for each tuple $(a_1,a_2,\ldots,a_{r-1})\in\bar{F}^{r-1}$, there are at most finitely many $\bar{F}$-isomorphism classes of Drinfeld modules $\phi$ with $J(\phi)=(a_1,a_2,\ldots,a_{r-1},1)$. 
\end{rem}

Now we define the {\em $J$-height} of $\phi$: 
\begin{equation}
h_J(\phi) := h(J) = \frac{1}{[K:F]}\sum_{v\in M_K} n_v \log\max\{ |j_1|_v, |j_2|_v, \ldots, |j_r|_v\},
\end{equation}
which is just the logarithmic Weil height of the tuple $J$. This height, and its local components $h_J^v(\phi)$ for $v\in M_K$, are invariant under isomorphisms of $\phi$.

In the special case $r=2$, we see that $j_1=j=\frac{g_1^{q+1}}{g_2}$ is the usual $j$-invariant and $h(J) = h(j)$ is the usual height of $j\in K$.

\medskip

Next, consider the weighted projective space $W\BP:=\BP(q-1,q^2-1,\cdots,q^r-1)$, which is Proj of the graded polynomial ring $K[g_1,g_2,\ldots,g_r]$, where the indeterminates are assigned the weights $\deg g_k = q^k-1$.

It is well-known that $W\BP-V(g_r=0)$ is the coarse moduli space of rank $r$ Drinfeld modules. Indeed, if $\phi'$ is another Drinfeld module with $\phi'_t(X) = tX + g'_1X^q + g'_2X^{q^2} + \cdots + g'_rX^{q^r}$, then
$(g_1,g_2,\ldots,g_r)$ and $(g'_1,g'_2,\ldots,g'_r)$ represent the same point in $W\BP$ if and only if $\phi$ and $\phi'$ are isomorphic over some algebraically closed field. 

One can define heights on weighted projective spaces in the obvious way, and the height associated to the point representing $\phi$ is called the {\em graded height} of $\phi$:
\begin{equation}
h_G(\phi) := \frac{1}{[K:F]}\sum_{v\in M_K} n_v\log\max\{|g_1|_v^{1/(q-1)},|g_2|_v^{1/(q^2-1)},\ldots,|g_r|_v^{1/(q^r-1)}  \}.
\end{equation}
For a finite place $v\in M_K^f$, the local component $h_G^v(\phi) = n_v\log\max_i |g_i|_v^{1/(q^i-1)}$ equals Taguchi's $v(\phi)$,
see \cite[\S2]{Tag}.

From the product formula, we see that
\begin{equation}\label{lcm}
dh_G(\phi) = h_J(\phi),
\end{equation}
and so again $h_G(\phi)$ is invariant under isomorphism. However, the local components $h_G^v(\phi)$ depend 
on the choice of $\phi$ in its isomorphism class. 

\begin{prop}\label{Northcott}
Let $F=\BF_q(t)$ and let $K/F$ be a finite extension. Let $C >0$. 
Then there are only finitely many $\bar{F}$-isomorphism classes of rank $r$ Drinfeld modules $\phi$ defined over $K$ such that $h_J(\phi) < C$ (respectively $h_G(\phi) < C$). 
\end{prop}

\begin{proof}
The usual Northcott Theorem for the Weil height implies that there are only finitely many $(j_1,j_2,\ldots,j_{r-1})\in K^{r-1}$ for which $h(j_1,j_2,\ldots,j_{r-1},1) < C$. The result now follows from Remark \ref{Potemine} and the identity $dh_G(\phi) = h_J(\phi)$.
\end{proof}

\section{Main results}\label{sec:main}

Let  $f : \phi \to \phi'$ be an isogeny of Drinfeld modules (still of generic characteristic) of degree $\deg f := \#\ker f$.  We may associate to $f$ a (not necessarily unique) dual isogeny $\hat{f} : \phi'\to\phi$ of degree $\deg\hat{f} \leq (\deg f)^{r-1}$, such that $\hat{f}\circ f = \phi_N$, where $N\in A$ is an element of minimal degree for which $\ker f \subset \phi[N]$, and similarly $f\circ \hat{f} = \phi'_N$. In particular, $\deg N = \frac{1}{r}\big(\log\deg f + \log\deg\hat{f}\big)\leq \log\deg f$.
 See for example \cite[Lemme 2.19]{DD}.

Denote by $\bar{K}$ an algebraic closure of $K$.
We now state our main result. 

\begin{thm}\label{main}
Let $f : \phi \to \phi'$ be an isogeny of rank $r$ Drinfeld modules over $\bar{K}$ and suppose $\ker f \subset \phi[N]$. 
\begin{enumerate}
\item[1.] We have
\begin{equation}
|h_G(\phi') - h_G(\phi)| \; \leq\;
\deg N + \left(\frac{q}{q-1} - \frac{q^r}{q^r-1}\right).
\end{equation}
\item[2.] Suppose $r=2$, then we have the following variant. We let $j=j_1(\phi)$ and $j'=j_1(\phi')$, then
\begin{equation}
h(j')-h(j) \leq \frac{q^2-1}{2}\log\deg f + \frac{q^2-1}{2}\log\left[1+\frac{1}{q}h(j')\right] + q.
\end{equation}
\end{enumerate}
\end{thm}

This is the analogue of Theorem 1.1. of \cite{Paz}.

\begin{rem}
If $r \geq 3$, we cannot hope to get a similar result replacing $h_G$ with the height $h(j_k)$ of a single invariant, as the following example shows.

Fix a Drinfeld module $\phi$ and consider all isogenies $f : \phi \to \phi'$ of kernel $\ker f \cong A/tA$, where
\[
\phi_t(X) = tX + g_1X^q + X^{q^3}; \qquad
\phi'_t(X) = tX + g'_1X^q + g'_2X^{q^2} + X^{q^3}; \qquad
f(X) = f_0X + X^q.
\]

From $f\circ\phi_t = \phi'_t \circ f$, comparing coefficients of $X^q$ and $X^{q^3}$, we obtain
\begin{equation}\label{gs}
g'_1 = f_0^{-q}(f_0g_1+t^q-t) \quad\text{and}\quad 
g'_2 = f_0 - f_0^{q^3},\quad\text{respectively.}
\end{equation}
Since $\ker f \subset \ker\phi_t$, we write $\phi_t = P \circ f$, for some $P(X) = aX + bX^q + X^{q^2}$. Comparing coefficients gives
\[
a = f_0^{-1}t, \qquad
b = -f_0^{q^2} \qquad\text{and}\qquad
f_0^{-1}t - f_0^{q^2+q} = g_1.
\]
Thus $f_0$ is a root of 
\begin{equation}\label{allisogs}
X^{q^2+q+1}+g_1X-t=0.
\end{equation}
Conversely, every root $f_0$ of (\ref{allisogs}) produces an isogeny $f : \phi \to \phi'$ as above. For each such root, we obtain from (\ref{gs}),
\begin{equation}\label{gs2}
g'_1 = -f_0^{q^2+1} + f_0^{-q}t^q; \qquad
g'_2 = f_0 -f_0^{q^3}.
\end{equation}

In particular, if $h(g_1)$ is very large, then at least one of the roots $f_0$ of (\ref{allisogs}) has large height, and thus so does the corresponding $g'_2$. Then $h(j_2(\phi'))$ is large, whereas $h(j_2(\phi))=0$ and $\deg f = q$.
\end{rem}

The following result follows from Theorem \ref{main} and \cite[Thm 1.3]{DD}:

\begin{cor}
There exists an effectively computable constant $C$, depending only on $r$ and $q$, such that the following holds.
Suppose $\phi$ and $\phi'$ are rank $r$ Drinfeld $\BF_q[t]$-modules, defined over a finite extension $K/F$, which are isogenous over $\bar{K}$. Then 
\[
|h_G(\phi')-h_G(\phi)| \leq  10(r+1)^7 \log\big([K:F] h_G(\phi)\big) + C.
\]
\end{cor}

\begin{proof}
By \cite[Thm 1.3]{DD}, there exists an effectively computable constant $c_2=c_2(r,q)$ and an isogeny 
$f:\phi \to \phi'$ of degree
\[
\deg(f) \leq c_2 \big([K:F] h(\phi)\big)^{10(r+1)^7}.
\]
Here $h(\phi)$ denotes a height function defined in terms of the coefficients of $\phi$ by $h(\phi) = \max\{h(g_1),\ldots,h(g_r)\}$. It is easy to see that $h(\phi) \leq (q^r-1)h_G(\phi)$, so
\begin{eqnarray*}
|h_G(\phi') - h_G(\phi)| & \leq & \deg N + \left(\frac{q}{q-1} - \frac{q^r}{q^r-1}\right) \\
& \leq & \log\deg f + \left(\frac{q}{q-1} - \frac{q^r}{q^r-1}\right) \\
& \leq & \log\big(c_2\big([K:F] (q^r-1)h_G(\phi)\big)^{10(r+1)^7}\big) + \left(\frac{q}{q-1} - \frac{q^r}{q^r-1}\right).
\end{eqnarray*}
The result follows.
\end{proof}

\bigskip

Applying Proposition \ref{Northcott}, we recover the following result, which was originally proved by Taguchi in \cite{Tag2}.

\begin{cor}\label{finiteness}
Each $\bar{K}$-isogeny class of Drinfeld modules defined over $K$ contains only finitely many $\bar{K}$-isomorphism classes of Drinfeld modules. 
\end{cor}

Note that our approach would lead the interested reader to an explicit bound on the number of $\bar{K}$-isomorphism classes within a $\bar{K}$-isogeny class.

\section{Lattices and Taguchi's height}\label{sec:Taguchi}

\subsection{Lattices}
Let $\Finf=\BF_q((\frac{1}{t}))$ be the completion of $F$ at the place $\infty$ and $\Cinf = \hat{\bar{F}}_\infty$ the completion of an algebraic closure of $\Finf$; it is complete and algebraically closed and plays the role of the complex numbers in characteristic $p>0$. Recall that $A=\BF_q[t]$.

A {\em lattice} of rank $r\geq 1$ is an $A$-submodule $\Lambda\subset\Cinf$ of the form $\Lambda = \omega_1A + \omega_2A + \cdots + \omega_rA$, where the $\omega_1,\omega_2,\ldots,\omega_r \in\Cinf$ are $\Finf$-linearly independent.

A {\em successive minimum basis} for a lattice $\Lambda$ is an $A$-basis $(\omega_1,\omega_2,\ldots,\omega_r)$ for $\Lambda$ satisfying the properties
\[
|\omega_1| \geq |\omega_2| \geq \cdots \geq |\omega_r|
\]
and
\[
|\omega_k| = \inf\big(\{ \lambda - \sum_{i=k+1}^r a_i\omega_i \;|\; a_{k+1},a_{k+2},\ldots,a_r \in A, \; \lambda\in\Lambda\}\smallsetminus\{0\}\big) \qquad\text{for $k=1,2,\ldots,r$}.
\]
In other words, $\omega_r$ is a minimal non-zero element of $\Lambda$ and each $\omega_k$ is minimal among the non-zero elements of $\Lambda$ not spanned by the $\omega_{k+1},\omega_{k+2},\ldots,\omega_r$. We can think of such a basis as being an ``orthogonal'' basis. Every lattice has a successive minimum basis, and we define the {\em covolume} of $\Lambda$ by
\begin{equation}
D(\Lambda) := |\omega_1| |\omega_2| \cdots |\omega_r|,
\end{equation}
where $(\omega_1,\omega_2,\ldots,\omega_r)$ is any successive minimum basis of $\Lambda$. By \cite[(4.1)]{Tag} or \cite[Prop. 3.1]{GekCrelle}, this is independent  of the choice of successive minimum basis.

The covolume of a lattice satisfies the following desirable properties.

\begin{lem}
Let $\Lambda \subset \Cinf$ be a lattice of rank $r$. 
\begin{enumerate}
	\item[1.] Choose an $\BF_q[t]$-basis $(\omega_1,\omega_2,\ldots,\omega_r)$ of $\Lambda$. Let $\gamma\in\GL_r(\Finf)$ and denote by
	 $\gamma\Lambda$ the lattice spanned by $(\omega_1,\omega_2,\ldots,\omega_r)\gamma^{\mathrm{T}}$. Then
	\[
	D(\gamma\Lambda) = |\det\gamma|D(\Lambda).
	\]
	\item[2.] Let $c\in\Cinf$, then 
	\[
	D(c\Lambda) = |c|^rD(\Lambda).
	\]
	\item[3.] Let $\Lambda'$ be a lattice of rank $r$ such that $\Lambda \subset \Lambda' \subset \Cinf$. Then
	\[
	(\Lambda' : \Lambda) = D(\Lambda)/D(\Lambda').
	\]
\end{enumerate}
\end{lem}

\begin{proof}
Part 1 is \cite[Prop. 4.4]{Tag} applied to the $A$-lattice $\Lambda$ inside the $\Finf$-vector space $V\subset\Cinf$ spanned by $(\omega_1,\omega_2,\ldots,\omega_r)$.

Part 2 follows from the definition and Part 3 follows from Part 1 as $\Lambda = \gamma\Lambda'$ for a suitable $\gamma\in\GL_r(F)$ with coefficients in $A$ and $|\det\gamma| = (\Lambda' : \Lambda)$.
\end{proof}

The lattice $\Lambda\subset\Cinf$ is said to be {\em reduced} if it has a successive minimum basis $(\omega_1,\omega_2,\ldots,\omega_r)$ with $\omega_r=1$. 
Equivalently, 
\begin{quote}
$\Lambda$ is reduced if and only if $1\in\Lambda$ and every non-zero $\lambda\in\Lambda$ satisfies $|\lambda|\geq 1$.
\end{quote}

Every Drinfeld module $\phi$ over $\Cinf$ is associated to a rank $r$ lattice $\Lambda\subset\Cinf$ and vice versa. We call the Drinfeld module $\phi$ {\em reduced} if its associated lattice is reduced. Every Drinfeld module is isomorphic over $\Cinf$ to a reduced Drinfeld module. (The analogous condition on an elliptic curve is to correspond to a point in the fundamental domain of the upper half-plane.)

\begin{lem}[Analytic Isogeny Lemma]\label{IsogAnalytic}
Let $f : \phi \to \phi'$ be an isogeny of reduced Drinfeld modules over $\Cinf$ with associated reduced lattices 
$\Lambda, \Lambda'\subset\Cinf$, respectively. Then 
\begin{equation}
-\log\deg\hat{f} \;\leq\; \log D(\Lambda) - \log D(\Lambda') \;\leq\; \log\deg f.
\end{equation}
\end{lem}

\begin{proof}
Analytically, the isogeny $f : \phi \to \phi'$ is given by multiplication by $\alpha\in\Cinf$ for which $\alpha\Lambda\subset\Lambda'$ and $\ker f \cong \Lambda'/\alpha\Lambda$. Thus
\[
\deg f = (\Lambda' : \alpha\Lambda) = D(\alpha\Lambda)/D(\Lambda') = |\alpha|^rD(\Lambda)/D(\Lambda'),
\]
so $D(\Lambda)/D(\Lambda') = |\alpha|^{-r}\deg f$. Since $\Lambda$ is reduced, $1\in\Lambda$ and thus $\alpha\cdot 1 \in\Lambda'$. Since $\Lambda'$ is reduced, we must have $|\alpha|\geq 1$, giving 
\[
D(\Lambda)/D(\Lambda') \leq \deg f.
\]
This proves the upper bound, and the lower bound follows by applying this to $\hat{f} : \phi' \to \phi$.
\end{proof}

\subsection{Taguchi's height}
Let $\phi$ be a Drinfeld module defined over a finite extension  $K/F$. We recall that $\phi$ is said to have {\em stable reduction} at a place $v\in M_K^f$ if it is isomorphic over $F$ to a Drinfeld module $\tilde\phi$ defined over the valuation ring $\CO_v\subset F$ of $v$ whose reduction modulo the maximal ideal $\Fm_v$ of $\CO_v$ is a Drinfeld module of positive rank over the residue field $\CO_v/\Fm_v$. Equivalently, $h_G^v(\phi)\in\BZ$, see \cite[p. 301]{Tag}. We say that $\phi$ has {\em everywhere stable reduction} if it has stable reduction at every finite place $v\in M_K^f$, equivalently if $h_G^v(\phi) \in\BZ$ for every $v\in M_K^f$.
By \cite[Lemme 2.10]{DD}, every Drinfeld module over $K$ acquires everywhere stable reduction after replacing $K$ by a finite extension thereof.

In \cite{Tag} Taguchi defines the {\em differential height} of $\phi$ as the degree of the metrised conormal line-bundle along the unit section associated to a minimal model of~$\phi$. It serves as the analogue of the Faltings height. 
All we need here is the identity (5.9.1) of \cite{Tag}, valid for Drinfeld modules with everywhere stable reduction, which we adopt as our definition:
\begin{eqnarray}
\hTag(\phi) & := & \frac{1}{[K:F]}\left[ \sum_{v\in M_K^f} h_G^v(\phi) - \sum_{v\in M_K^\infty} n_v \log D(\Lambda_v)^{1/r}\right] \\
& = & \hTag^f(\phi) + \hTag^\infty(\phi), \nonumber
\end{eqnarray}
where we pose $\displaystyle{\hTag^f(\phi)=\frac{1}{[K:F]} \sum_{v\in M_K^f} h_G^v(\phi)}$ and  $\displaystyle{\hTag^\infty(\phi)=- \frac{1}{[K:F]}\sum_{v\in M_K^\infty} n_v \log D(\Lambda_v)^{1/r}}$.

Notice that the sign is included in the definition of $\hTag^\infty(\phi)$ - the reader should keep this in mind when reading the remaining calculations in this paper.

Here $\Lambda_v \subset\BC_\infty$ is the lattice associated to the Drinfeld module $\phi^\sigma$ over $\BC_\infty$ obtained by embedding the coefficients $g_k$ into $\BC_\infty$ via the embedding $\sigma : K \into \BC_\infty$ associated to the infinite place $v$. 

We see that the finite part coincides with the finite part of our graded height,
\[
\hTag^f(\phi) = h_G^f(\phi).
\]

Our definition of $\hTag(\phi)$ only coincides with Taguchi's differential height when $\phi$ has stable reduction at every finite place. For the general case, the reader will find an excellent treatment of Taguchi's height in \cite[\S5.1]{Wei}. 
We add a proof that $\hTag(\phi)$ as defined above satisfies the following desirable properties.
\begin{lem}\label{TagHeightInvariance}
Let $\phi$ be a Drinfeld module with everywhere stable reduction, defined over a global function field~$K$.
\begin{enumerate}
	\item[1.] $\hTag(\phi)$, $\hTag^f(\phi)$ and $\hTag^\infty(\phi)$ do not depend on the choice of the field~$K$.
	\item[2.] $\hTag(\phi)$ is invariant under $\bar{K}$-isomorphism.
\end{enumerate}
\end{lem}

\begin{proof}
Suppose that $L/K$ is a finite extension. Suppose $v_1,v_2\in M_L$ lie above the same place of $K$. Then $|g_i|_{v_1} = |g_i|_{v_2}$ for each $i$ and also $D(\Lambda_{v_1}) = D(\Lambda_{v_2})$. The first point follows as usual from $\displaystyle{[L:K] = \sum_{v|w}[L_v:K_w]}$.

To prove the second point, let $c\in\bar{K}^*$ and replace $K$ by $K(c)$, which we may by the first point. Now
\[
\hTag^f(c^{-1}\phi c) = \hTag^f(\phi) + \frac{1}{[K:F]}\sum_{v\in M_K^f}n_v \log|c|_v,
\]
and
\begin{eqnarray*}
\hTag^\infty(c^{-1}\phi c) & = & -\frac{1}{[K:F]} \sum_{v\in M_K^\infty} n_v\log D(c^{-1}\Lambda_v)^{1/r} \\
& = & -\frac{1}{[K:F]} \sum_{v\in M_K^\infty} n_v\log [|c|^{-r} D(\Lambda_v)]^{1/r} \\
& = & \hTag^\infty(\phi) + \frac{1}{[K:F]}\sum_{v\in M_K^\infty}n_v \log|c|_v.
\end{eqnarray*}
The result now follows from the product formula $\displaystyle{\sum_{v\in M_K} n_v\log|c|_v = 0}$.
\end{proof}

The advantage of $\hTag$ is that it behaves well under isogenies.

\begin{lem}[Taguchi's Isogeny Lemma]\label{TaguchiIsog}
Let $f : \phi \to \phi'$ be a $\bar{K}$-isogeny between two rank $r$ Drinfeld modules over $K$ with everywhere stable reduction. Then 
\[
-\frac{1}{r}\log\deg\hat{f} \; \leq \; \hTag(\phi')-\hTag(\phi) \;\leq\; \frac{1}{r}\log\deg f.
\]
\end{lem}

\begin{proof}

We start with Lemma 5.5 in \cite{Tag}, it states that 
\[
\hTag(\phi') - \hTag(\phi) = \frac{1}{r}\log\deg(f) - \frac{1}{[K:F]}\log\#(R/D_f).
\]
Here, $R$ is the integral closure of $A$ in $F$ and the ideal $D_f \subset R$ is the {\em different} of $f$. We don't need the exact definition of this, merely the fact that $\#(R/D_f)$ is a positive integer, so $\log\#(R/D_f)\geq 0$. This gives us the upper bound, and the lower bound is obtained by applying the upper bound to the dual isogeny $\hat{f} : \phi' \to \phi$.
\end{proof}

\section{Analytic estimates}\label{sec:analytic}

The proof of Theorem \ref{main} involves breaking up the difference in heights, using the identity $\hTag^f(\phi) = h_G^f(\phi)$, as follows.
\begin{equation}\label{MainIdentity}
h_G(\phi')-h_G(\phi) = 
\underbrace{\big[\hTag(\phi')-\hTag(\phi)\big]}_{(A)} 
+ \underbrace{\big[h_G^\infty(\phi') - h_G^\infty(\phi)\big]}_{(B)}
+ \underbrace{\big[\hTag^\infty(\phi) - \hTag^\infty(\phi')\big]}_{(C)}.
\end{equation}

Part (A) is bounded using Taguchi's Isogeny Lemma \ref{TaguchiIsog}. 

Bounding the terms (B) and (C) will require some analytic estimates, which we outline next.

\subsection{Proof of Theorem \ref{main}, Part 1.}

We start by \cite[Lemme 2.10]{DD}: we may replace $K$ by a finite extension so that $\phi$ and $\phi'$ have everywhere stable reduction. From now on, our Drinfeld modules are all assumed to have everywhere stable reduction.

\begin{lem}\label{inv}
Let $\phi$ be a Drinfeld module of rank $r$ over $\Cinf$ with associated lattice $\Lambda$. Then the quantity 
\[
\log\max\{ |g_1|^{1/(q-1)}, |g_2|^{(1/(q^2-1)}, \ldots, |g_r|^{1/(q^r-1)}\} + \log D(\Lambda)^{1/r} \in\BR
\]
is invariant under isomorphisms of $\phi$.
\end{lem}

\begin{proof}
Let $\phi' = c^{-1}\phi c$ with $c\in\Cinf^*$ be another Drinfeld module isomorphic to $\phi$. Then
\begin{eqnarray*}
\lefteqn{\log\max\{ |g'_1|^{1/(q-1)}, |g'_2|^{(1/(q^2-1)}, \ldots, |g'_r|^{1/(q^r-1)}\} + \log D(\Lambda')^{1/r} } \\
& = & \log\max\{ |c^{q-1}g_1|^{1/(q-1)}, |c^{q^2-1}g_2|^{(1/(q^2-1)}, \ldots, |c^{q^r-1}g_r|^{1/(q^r-1)}\} + \log D(c^{-1}\Lambda)^{1/r} \\
& = & \log\max\{ |g_1|^{1/(q-1)}, |g_2|^{1/(q^2-1)}, \ldots, |g_r|^{1/(q^r-1)}\}+ \log D(\Lambda)^{1/r}. \\
\end{eqnarray*}
\end{proof}
Let $\phi/K$ be a rank $r$ Drinfeld module with coefficients $g_1,\ldots,g_r\in K$.
To each infinite place $v\in M_K^\infty$ we associate an embedding $\sigma : K \into \Cinf$ for which $|x|_v=|x^\sigma|$ for any $x\in K$.

Then $\Lambda_v \subset\Cinf$ is the lattice associated to the Drinfeld module $\phi^\sigma/\Cinf$ defined by the coefficients $g_1^\sigma,\ldots,g_r^\sigma \in \Cinf$. 

We rewrite (B) + (C) of (\ref{MainIdentity}) as follows.
\begin{eqnarray}
\label{BCpart}
(B)+(C)  
& = & \frac{1}{[K:F]}\sum_{\sigma : K \into \Cinf}n_\sigma\left(\left[
\log\max_{1\leq i \leq r} |g'^\sigma_i|^{1/(q^i-1)} - \log\max_{1\leq i \leq r} |g_i^\sigma|^{1/(q^i-1)}
\right]\right. \label{BC}\\
& & + \left.\left[
\frac{1}{r}\log D(\Lambda'_\sigma) - \frac{1}{r}\log D(\Lambda_\sigma)
\right] \right) \nonumber.
\end{eqnarray}

By Lemma \ref{inv}, the term for each $\sigma : K \into \Cinf$ in (\ref{BC}) depends only on the isomorphism classes of $\phi^\sigma$ and $\phi'^\sigma$. 
Therefore, in the remainder of this section, we will frequently make the following reduction:

\begin{red}\label{Assumption}
{\it 
Whenever the Drinfeld module $\phi^\sigma$ arises in the context of (\ref{BC}), we replace it by an isomorphic reduced Drinfeld module, which we may by Lemma \ref{inv}, and which by abuse of notation we again denote by $\phi^\sigma$.
}
\end{red}

%
%

Under Reduction \ref{Assumption}, (C) is bounded by Lemma \ref{IsogAnalytic}:

\begin{equation}\label{BoundC}
\hTag^\infty(\phi) - \hTag^\infty(\phi') \leq \frac{1}{[K:F]}\sum_{v\in M_K^\infty}n_v\frac{1}{r}\log\deg\hat{f} = \frac{1}{r}\log\deg\hat{f}.
\end{equation}

Next, we obtain an absolute bound on part (B).

\begin{lem}\label{AbsBound}
Let $\phi/\Cinf$ be a reduced Drinfeld module of rank $r$. Then
\begin{equation}\label{eqAbsBound}
\frac{q^r}{q^r-1} \;\leq\; \log\max_{1\leq i \leq r}|g_i|^{1/(q^i-1)}    \;\leq\; \frac{q}{q-1}.
\end{equation}
\end{lem}

\begin{proof}
For this we must recall some concepts introduced in \cite{Gek1}. Define
\begin{eqnarray*}
\CF & := & \Big\{ (\omega_1,\omega_2,\ldots,\omega_r) \in \Cinf^r \;|\; \text{$\omega_r=1$ and $(\omega_1,\omega_2,\ldots,\omega_r)$ forms }\\
& & \text{\quad a successive minimum basis for the lattice $\omega_1A + \omega_2A + \cdots + \omega_rA$} \Big\}.
\end{eqnarray*}
This set is a fundamental domain (in a suitable sense) for the action of $\GL_r(A)$ on the Drinfeld period domain 
\[
\Omega^r =  \Big\{ (\omega_1,\omega_2,\ldots,\omega_r) \in \Cinf^r \;|\; \text{$\omega_r=1$ and $\omega_1,\omega_2,\ldots,\omega_r$ are linearly independent over $\Finf$} \Big\}.
\]
Every reduced Drinfeld module $\phi/\Cinf$ corresponds to a reduced lattice of the form $\Lambda = \omega_1A + \omega_2A + \cdots + \omega_rA$ for some $(\omega_1,\omega_2,\ldots,\omega_r)\in\CF$.

Denote by $\BT$ the Bruhat-Tits building of $\PGL_r(\Finf)$ and by $\BT(\BQ)$ the points in the realisation of $\BT$ with rational barycentric coordinates. 
The image of $\CF$ under the building map $\lambda : \Omega^r \to \BT(\BQ)$ (see \cite[\S2.3]{Gek1}) is an $(r-1)$-dimensional simplicial complex $W$, whose vertices correspond to integer $r$-tuples $(k_1,k_2,\ldots,k_r)\in\BZ_{\geq 0}^r$ with $k_1\geq k_2 \geq \cdots \geq k_r=0$. The preimage of such a vertex consists of lattice bases $(\omega_1,\omega_2,\ldots,\omega_r)$ satisfying $|\omega_i|=q^{k_i}$ for each $i$.

The origin of $W$ is denoted $\mathbf{o}=(0,0,\ldots,0)$. By \cite[\S4.6]{Gek1}, for $\omega\in\lambda^{-1}(\mathbf{o})$ we have $\log|g_r(\omega)|=q^r$ and for each $i=1,2,\ldots,r-1$, $\log|g_i(\omega)|\leq q^i$, with equality achieved somewhere on the set $\lambda^{-1}(\mathbf{o})$ by \cite[Cor. 4.16]{Gek1}. 

By \cite[Cor. 4.11 and 4.16]{Gek1} it follows that each $\log|g_i(\omega)|$ is non-increasing as $\lambda(\omega)$ moves away from $\mathbf{o}$ in $W(\BQ)$, and so, for every $i=1,2,\ldots,r$, 
\begin{equation}\label{gUpperBound}
\log |g_i(\omega)| \leq q^i \qquad\text{for all $\omega\in\CF$.}
\end{equation}
This implies the upper bound in (\ref{eqAbsBound}).

For each $i=1,2,\ldots,r-1$ we define the $i$th wall of $W$ to be the subcomplex $W_i\subset W$ spanned by vertices satisfying $k_i=k_{i+1}$. Its preimage under the building map is denoted 
\[
\CF_i = \lambda^{-1} \big(W_i(\BQ)\big) = \{(\omega_1,\omega_2,\ldots,\omega_r)\in\CF \;|\; |\omega_i|=|\omega_{i+1}|\}. 
\]

To prove the lower bound, first note that by \cite[Cor. 4.16]{Gek1}, 
\begin{equation}\label{g1LowerBound}
\log |g_1(\omega)|=q \qquad\text{for all } \omega\in\CF\smallsetminus \CF_{r-1}
\end{equation}
and we claim that, for $i=2,3,\ldots,r-1$, 
\begin{equation}\label{gLowerBound}
\log |g_i(\omega)| = q^i \qquad\text{for all } \omega\in \big(\CF_{r-1}\cap \CF_{r-2} \cap \cdots \cap \CF_{r-i+1}\big) \smallsetminus \CF_{r-i}.
\end{equation}

Indeed, by \cite[Cor. 4.16]{Gek1}, since $\omega\notin\CF_{r-i}$, $\log|g_i(\omega)|$ is constant on the fibres of $\lambda$, we may consider $\log|g_i|$ as a function on $W(\BQ)$. Since $1=|\omega_r|=|\omega_{r-1}|=\cdots=|\omega_{r-i+1}|$, the point $\lambda(\omega)\in W(\BQ)$ lies in a simplex all of whose edges can be reached from the vertex $\mathbf{o}$ by paths consisting entirely of edges of the form $\mathbf{k}\to\mathbf{k}+\mathbf{k}_\ell$ for $\ell\leq r-i$ (here $\mathbf{k}_\ell = (1,1,\ldots,1,0,\ldots,0)$ contains $\ell$ ones). By \cite[Prop. 4.10 and Cor. 4.16]{Gek1}, $\log|g_i|$ is constant on these edges, and it interpolates linearly within each simplex, hence $\log|g_i(\omega)| = \log\|g_i(\lambda^{-1}(\mathbf{o}))\| = q^i$, by \cite[\S4.6]{Gek1}. This proves the claim.

Every $\omega\in\CF$ lies in one of the subsets $$\CF\smallsetminus\CF_{r-1},\; \CF_{r-1}\smallsetminus \CF_{r-1}, \; \big(\CF_{r-1}\cap\CF_{r-2}\big)\smallsetminus\CF_{r-3},\ldots,\; \big(\CF_{r-1}\cap\cdots\cap\CF_{1}\big)=\{\mathbf{o}\}.$$ Hence, by (\ref{g1LowerBound}), (\ref{gLowerBound}) and $\log|g_r(\mathbf{o})|=q^r$, we have $\log|g_i(\omega)|=q^i$ for some $i=1,2,\ldots,r$, and the lower bound in (\ref{eqAbsBound}) follows.
\end{proof}

In particular, we find that in (\ref{MainIdentity}), after Reduction \ref{Assumption},
\begin{equation}\label{BoundB}
\big|h_G^\infty(\phi') - h_G^\infty(\phi)\big| \leq \frac{1}{[K:F]}\sum_{v\in M_K^\infty}n_v \left(\frac{q}{q-1} - \frac{q^r}{q^r-1}\right) = \frac{q}{q-1} - \frac{q^r}{q^r-1}.
\end{equation}

Now Lemma \ref{TaguchiIsog} together with (\ref{BoundC}) and (\ref{BoundB}) and the fact that $\deg N = \frac{1}{r}\big(\log\deg f + \log\deg\hat{f}\big)$ imply Theorem \ref{main}, Part 1. \qed

\subsection{Proof of Theorem \ref{main}, Part 2.}

\begin{lem}\label{jgrowth}
Let $\phi/\Cinf$ be a reduced rank $2$ Drinfeld module with associated reduced lattice $\Lambda\subset\Cinf$. Then
\[
1\leq D(\Lambda) \leq \max\left\{\frac{1}{q}\log|j(\phi)|,1\right\}.
\]
\end{lem}

\begin{proof}
We use an estimate of $|j(\phi)|$ obtained by Gekeler in \cite{Gek0}. The reduced rank 2 lattice $\Lambda$ has a successive minimum basis $(\omega_1,1)$, where $\omega_1\in\Cinf$ satisfies 
\[
1 \leq |\omega_1|=|\omega_1|_i := \inf_{x\in\Finf}|\omega_1 - x| = D(\Lambda).
\]
Suppose first that $k=\log D(\Lambda)\in\BZ_{\geq 0}$, then \cite[Theorem 2.17]{Gek0} gives
\[
\begin{array}{ll} \log|j(\phi)| = q^{k+1} = qD(\Lambda) & \text{if $k \geq 1$,} \\
\log|j(\phi)|\leq q = qD(\Lambda) & \text{if $k=0$.}\end{array}
\]
Furthermore, $\log |j(\phi)|$ interpolates linearly between integral values of $k$ \cite[Rem. 2.14]{Gek0}, in other words, if $k=\lfloor \log D(\Lambda) \rfloor$ and
$s = \log D(\Lambda) - k > 0$, then 
\[
\log|j(\phi)| = s q^{k+1} + (1-s) q^{k+2}.
\]
Since $x\mapsto q^{x+1}$ is convex, it follows that for $D(\Lambda) > 1$,
\[
\log|j(\phi)| \geq qD(\Lambda).
\]
The lemma follows.
\end{proof}

We now use this to get another estimate of (C) in the case $r=2$ and $j=j_1(\phi)$ and $j'=j_1(\phi')$. Since we are assuming that each $\Lambda_v$ is reduced, $D(\Lambda_v)\geq 1$.
We obtain

\begin{eqnarray*}
& & \frac{1}{[K:F]}\sum_{v\in M_K^\infty} n_v\big[\log D(\Lambda'_v)^{1/2} - \log D(\Lambda_v)^{1/2}\big] \\ 
& \leq & \frac{1}{[K:F]}\sum_{v\in M_K^\infty} n_v\log D(\Lambda'_v)^{1/2} \\
& \leq & \frac{1}{2[K:F]}\sum_{v\in M_K^\infty} n_v\log \max\left\{\frac{1}{q}\log|j'|_v,1\right\} \\
& = & \frac{1}{2[K:F]}\sum_{\sigma : K \into \Cinf} \log \max\left\{\frac{1}{q}\log|j'^\sigma|^{n_v},1\right\} \\
& = &\frac{1}{2} \log \left[\prod_{\sigma : K \into \Cinf} \max\left\{\frac{1}{q}\log|j'^\sigma|^{n_v},1\right\} ^{1/[K:F]}\right] \\
& \leq & \frac{1}{2} \log\left[ \frac{1}{[K:F]} \sum_{\sigma : K \into \Cinf} \max\left\{\frac{1}{q}\log|j'^\sigma|^{n_v},1\right\} \right] \qquad\text{(by the AM-GM inequality)} \\
& \leq & \frac{1}{2}\log\left[1+ \frac{1}{q}\frac{1}{[K:F]} \sum_{\sigma : K \into \Cinf} n_v\max\{\log|j'^\sigma|,0\} \right] \\
& \leq & \frac{1}{2} \log\left[1 + \frac{1}{q}h(j')\right].
\end{eqnarray*}
Plugging this, Lemma \ref{TaguchiIsog} and (\ref{BoundB}) into (\ref{MainIdentity}), we obtain
\[
h_G(\phi')-h_G(\phi) \leq \frac{1}{2}\log\deg f + \left(\frac{q}{q-1} - \frac{q^2}{q^2-1}\right) + \frac{1}{2} \log\left[1 + \frac{1}{q}h(j')\right].
\]
Finally, since $h(j) = (q^2-1)h_G(\phi)$, we obtain Theorem \ref{main}, Part 3, after multiplying by $(q^2-1)$. \qed

\section{Drinfeld modular polynomials} \label{sec:modular}

Let $m\in\BF_q[t]$ be monic. 
We define
\[
\psi(m)=\vert m\vert\prod_{P\vert m}\left(1+\frac{1}{\vert P\vert}\right) \quad \mathrm{and} \quad \kappa(m)=\sum_{P\vert m}\frac{\deg P}{\vert P\vert},
\]
where $P$ ranges over all monic irreducible factors of $m$. 

In analogy to classical modular polynomials, Bae \cite{Bae92} constructed polynomials $\Phi_m(X,Y)\in\BF_q[t][X,Y]$ for each monic $m\in\BF_q[t]$, called {\em Drinfeld modular polynomials}, with the following properties. 
\begin{enumerate}
	\item Degree: $\Phi_m(X,Y)$ is monic of degree $\psi(m)$ in each variable,
	\item Symmetry: $\Phi_m(X,Y)=\Phi_m(Y,X)$,
	\item Irreducibility: $\Phi_m(X,Y)$ is irreducible in $\Cinf[X,Y]$,
	\item Isogeny: $\Phi_m(j,j')=0$ if and only if $j=j(\phi)$ and $j'=j(\phi')$ are the $j$-invariants of rank two Drinfeld modules $\phi$ and $\phi'$ linked by an isogeny of kernel $A/mA$.
\end{enumerate}

To study the coefficients of $\Phi_m(X,Y)$, we introduce yet another height. To a polynomial $f$ in several variables with coefficients in $\mathbb{C}_\infty$, we associate its na\"ive height: 
\[
h(f)= \log\max_c|c|_\infty,
\]
where $c$ ranges over all the coefficients of $f$.

Hsia proved the following asymptotic result in \cite{hsia} page 237:

\begin{thm}[Hsia]\label{Hsia}
For any monic, non-constant polynomial $m\in{\BF_q[t]}$ we have 
$$h(\Phi_m)=\frac{q^2-1}{2}\psi(m) (\deg m -2\kappa(m) +O(1))$$ when $\vert m\vert$ tends to infinity.
\end{thm}

Our goal in this last section is to give a completely explicit upper bound for $h(\Phi_m)$. We start by preparing an interpolation lemma with the following set of interpolation points.

\begin{lem}\label{S}
Let $n\geq0$ be an integer. Consider the set 
\[
S_n=\Big\{\alpha_n t^n+\cdots + \alpha_0+\cdots + \alpha_{-n}t^{-n}\;\vert \; \forall i \in{\{-n, \ldots, n\}}, \; \alpha_i\in{\mathbb{F}_q}\Big\}.
\]
It has cardinality $q^{2n+1}$. Let $d\leq q^{2n+1}-1$, and consider $d+1$ distinct points $y_0,y_1,\ldots,y_d\in S_n$.

For any $k\in{\{0, \cdots,d\}}$, denote 
$\displaystyle{T_k(Y)=\prod_{\substack{s=0\\s\neq k}}^d (Y-y_s)=\sum_{j=0}^d a_j Y^j}$, 
where the product is taken over all $s$ in $\{0, \cdots,d\}$ different from $k$. Then we have 
\begin{enumerate}
\item $\max\Big\{\vert a_j \vert_\infty \, \Big\vert \; j\in\{0,\ldots, d\}\Big\}\leq q^{nd}$,
\item $\displaystyle{\prod_{\substack{s=0\\s\neq k}}^{d}\vert y_k-y_s\vert_\infty \geq q^{-nd}}.$
\end{enumerate}
\end{lem}

\begin{proof}
The maximum degree in $t$ of elements in $S_n$ is $n$, the upper bound then comes from the explicit computation of the coefficients of $T_k$ in terms of elements of $S_n$, and the degree of $T_k$ is $d$. The minimum degree in $t$ of a non-zero difference of elements in $S_n$ is $-n$, the lower bound is direct as well. Note that in \cite[Lemma 5.1]{hsia}, the interpolation set of points is chosen with the extra property $\vert y_k-y_s\vert_\infty = \max\{\vert y_k\vert_\infty, \vert y_s\vert_\infty\}$, which is not assumed here.
\end{proof}

\begin{lem}\label{lem20}
Let $P\in{\mathbb{C}_\infty[X,Y]}$ be a nonzero polynomial of degree at most $d\geq1$ in each variable. Suppose there exists a real number $B>0$ such that 
 $h(P(X,y_k))\leq B$ for each $y_k$ in the set $S_n$ defined in Lemma \ref{S}. Then we have 
\begin{equation}
h(P)\leq B+2nd.
\end{equation}
\end{lem}

\begin{proof}
We may write $P(X,Y)=\displaystyle{\sum_{0\leq r\leq d}Q_r(Y)X^r}$ for some polynomials $Q_r(Y)\in\Cinf[Y]$. For any degree $0\leq r\leq d$ and any of the above points $y_k$, let $c_{k,r}=Q_r(y_k)$ be the coefficient of $X^r$ of the polynomial $P(X,y_k)$. By Lagrange interpolation, one has 
\begin{equation}
Q_r(Y)=\sum_{k=0}^{d}c_{k,r}\prod_{\substack{s=0\\s\neq k}}^{d}\frac{Y-y_s}{y_k-y_s}.
\end{equation}
We write $T_k(Y)=\displaystyle{\prod_{\substack{s=0\\s\neq k}}^{d}(Y-y_s)}$, by Lemma \ref{S} we have $h(T_k)\leq q^{nd}$ and $\displaystyle{\prod_{\substack{s=0\\s\neq k}}^{d}\vert y_k-y_s\vert_\infty \geq q^{-nd}}$ for any $k\in\{0,\cdots,d\}$, and by assumption $|c_{k,r}|_\infty\leq B$. The result follows.
\end{proof}

We add a small technical lemma.

\begin{lem}\label{small}
Let $a$ be a positive real number. Let $q\geq2$ be a prime power. Assume $x\geq q^{3}$. Then $\ln(1+\frac{x}{q})\leq\frac{x}{q^2}$ and the inequality 
\begin{equation}
x\leq a+\frac{q^2-1}{2}\log\Big(1+\frac{x}{q}\Big)
\end{equation}
implies
\begin{equation}
x\leq a+\frac{q^2-1}{2}\log\left(1+\frac{a}{q}\Big(1-\frac{q^2-1}{2q^2\ln q}\Big)^{-1}\right),
\end{equation}
where $\ln$ is the natural logarithm and $\log$ is the logarithm in base $q$.
\end{lem}

\begin{proof}
Direct computation.
\end{proof}

We are now ready to prove the following.

\begin{prop}\label{Drinfeld modular}
For any monic, non-constant polynomial $m\in{\mathbb{F}_q[t]}$, the height $h(\Phi_m)$ is bounded above by
$$\psi(m)\max\left(q^{3},\;\Big[ \frac{q^2-1}{2}\deg m + q+\frac{1}{2}(\log\psi(m)+1)+ \frac{q^2-1}{2}\log\left(1+\frac{a}{q}\Big(1-\frac{q^2-1}{2q^2\ln q}\Big)^{-1}\right)  \Big]\right)$$
$$+2\psi(m)\log\psi(m),$$
where $a=\frac{q^2-1}{2}\deg m + q+\frac{1}{2}(\log\psi(m)+1)$.
\end{prop}

\begin{proof}
Let us fix $j_0\in{S_n}$, where $n$ is chosen such that $q^{2n+1}\geq\psi(m)+1\geq q^{2n-1}$, and $S_n$ is the set of Lemma \ref{S}. The relation between roots and coefficients for the polynomial $\Phi_m(X,j_0)$ gives in particular the inequality $$h(\Phi_m(X,j_0))\leq \psi(m) \max_j h(j),$$ where the maximum is taken over all the roots $j$ of $\Phi_m(X,j_0)$. Each of these roots corresponds to a Drinfeld module isogenous to the fixed one corresponding to $j_0$, hence by Theorem \ref{main}, we get 
\begin{equation}
h(j)-h(j_0) \leq \frac{q^2-1}{2}\deg m + \frac{q^2-1}{2}\log\left(1+\frac{1}{q}h(j)\right) + q.
\end{equation}
Now for any $j_0\in{S_n}$, we have $h(j_0)\leq n\leq \frac{1}{2}(\log \psi(m)+1)$. This leads to:

\begin{equation}
h(j)\leq \frac{q^2-1}{2}\deg m + q+\frac{1}{2}(\log\psi(m)+1) + \frac{q^2-1}{2}\log\left(1+\frac{1}{q}h(j)\right).
\end{equation}
Assume $h(j)\geq q^{3}$, then by Lemma \ref{small} we get
\begin{equation}
h(j)\leq \frac{q^2-1}{2}\deg m +q+\frac{1}{2}(\log\psi(m)+1)+ \frac{q^2-1}{2}\log\left(1+\frac{a}{q}\Big(1-\frac{q^2-1}{2q^2\ln q}\Big)^{-1}\right),
\end{equation}
where $a=\frac{q^2-1}{2}\deg m +q+\frac{1}{2}(\log\psi(m)+1)$ and hence $h(\Phi_m(X,j_0))$ is bounded above by a quantity $B$ equal to
\begin{equation}
\psi(m)\max\left(q^{3},\;\Big[ \frac{q^2-1}{2}\deg m + q+\frac{1}{2}(\log\psi(m)+1)+ \frac{q^2-1}{2}\log\left(1+\frac{a}{q}\Big(1-\frac{q^2-1}{2q^2\ln q}\Big)^{-1}\right)  \Big]\right),
\end{equation}
and by Lemma \ref{lem20} we obtain the result.

\end{proof}

Asymptotically, this gives
\begin{equation}\label{Drinfeld modular asymp}
h(\Phi_m) < \left(\frac{q^2+4}{2}+\epsilon\right)\psi(m)\deg(m)
\end{equation}
for $\deg(m)$ sufficiently large compared to $\epsilon>0$. This is only slightly weaker than Hsia's exact asymptotic in Theorem \ref{Hsia}. 

Another completely explicit upper bound on $h(\Phi_m)$, of order $(q/2)|m|\psi(m)^2$, was obtained by Bae and Lee in \cite[Theorem 3.7]{BL97}.
\\

\end{document}